\documentclass{birkjour}
\usepackage{amsmath}
 \newtheorem{thm}{Theorem}[section]
 
 \newtheorem{lem}[thm]{Lemma}
 
 \theoremstyle{definition}
 \newtheorem{defn}[thm]{Definition}
 \theoremstyle{remark}
 \newtheorem{rem}[thm]{Remark}
 
 \numberwithin{equation}{section}

\begin{document}

%
%
%
%
%
%
%
%
%

\title[Finiteness of CCs for Homogeneous Potential]{Finiteness of non-degenerate central configurations of the planar $n$-body problem with a homogeneous potential}

\author{Julius NATRUP}
\address{
Institute of Mathematics\\
University of Augsburg\\
Universitätsstraße 2\\
86159 Augsburg\\
Germany
}
\email{julius.natrup@math.uni-augsburg.de}
\thanks{This work is supported by the DFG grant ZH 605/1-1 and the NSFC Young Scientists Fund 11901160.}

\author{Qun WANG}

\address{%
Department of Mathematical and Computational Sciences\\
University of Toronto Mississauga\\
3359 Mississauga Road\\
Mississauga, ON, L5L 1C6\\
Canada}
\email{qun.wang@utoronto.ca}

\author{Yuchen WANG}
\address{School of Mathematical Science\br
Tianjin Normal University\br
Tianjin, 300387\br
China}
\email{ycwang@tjnu.edu.cn}


\date{January 1, 2004}


\begin{abstract}
We show that there exist an upper bound and a lower bound for the number of non-degenerate central configurations of the n-body problem in the plane with a homogeneous potential. In particular, both bounds are independent of the homogeneous degree of the potential under consideration.   
\end{abstract}

\maketitle
\section{Introduction}
The movements of $n$ point masses in the plane under the Newtonian gravity are governed by the equations
\begin{align}
    m_i\ddot{q}_{i}(t) =\sum_{\substack{ 1\leq j\leq n\\j\neq i}}m_i m_j\frac{ q_{ij}(t) }{\|q_{ij}(t)\|^3} =\nabla_{i}U(q_1,q_2,...,q_n),  \quad 1\leq i\leq n,
\end{align}
where $m_i\in\mathbb{R}_{+}$ and $q_{i}= (x_i,y_i)\in \mathbb{R}^2$ are the mass and the position in the plane of the $i^{th}$ body respectively, and $q_{ij}=q_j-q_i$. Here $\|q_{ij}\|$ stands for the Euclidean distance between the $i^{th}$ body and the $j^{th}$ body. The function  
\begin{align}
U:\mathbb{R}^{2n}\setminus \Delta \rightarrow \mathbb{R},\quad    U(q_1,q_2,...,q_n) = \sum_{1\leq i<j\leq n}\frac{m_im_j}{\|q_{ij}\|}, 
\end{align}
is the Newtonian force function, where $\Delta$ is the collision set, defined by 
\[\Delta = \{(q_1,q_2,...,q_n)\in\mathbb{R}^{2n}, \exists 1\leq i<j\leq n, q_i=q_j\}.\]
\begin{defn}
 A \textit{central configuration} $\mathbf{q}=(q_1,q_2,...,q_n)$  is a configuration satisfying 
\begin{align}
\label{eq:cc}
    -\nabla_i U(q_1,q_2,\dots,q_n) = \lambda m_i(q_i-c), \quad 1\leq i\leq n. 
\end{align}   
\end{defn}
\sloppy {In the above system, $\lambda\in\mathbb{R}$ is uniquely determined by $\lambda = U / I$, where $I =\sum_{i=1}^{n}m_i \|q_i-c\|^2$ is the inertia, and $c=\sum_{i=1}^{n}m_iq_i / \sum_{i=1}^{n}m_i$ is the center of mass.}

Note that the system \eqref{eq:cc} determining the central configurations is invariant under the (diagonal) action of the Euclidean group $E(2)$. As a result, any central configuration is not isolated, and hence belongs to an equivalence class of central configurations. We may fix $c$ to be $0$ using the translation invariance, and normalise $\lambda$ to be $1$ using the dilation invariance. The equation \eqref{eq:cc} becomes    
\begin{align}
\label{eq:cc1}
     q_{i} =\sum_{\substack{ 1\leq j\leq n\\j\neq i}}m_i \frac{q_{ij} }{\|q_{ij}\|^3},  \quad 1\leq i\leq n.
\end{align}

The central configurations permit explicit construction of self-similar solutions, serve as the limit of collision configuration after blow-up, and produce new solutions through perturbation. They play a significant role in understanding the dynamics of the $n$-body problems. However basic questions concerning the central configurations, among which their finiteness,  is far from completely known yet. 

The finiteness of central configurations of the $3$-body problem consists of only three Eulerian (collinear) and two Lagrangian (equilateral triangle) configurations. The finiteness of the 4-body problem in the plane is much more challenging. In \cite{moeckel2001}, Moeckel proved that there are only finitely many central configurations in the planar four-body problem for almost all choices of masses. The finiteness for all the positive masses is proved by Hampton and Mockel \cite{hampton2006finiteness}. Albouy and Kaloshin \cite{albouy2012finiteness} provided another proof for the finiteness of the 4-body problem. In the same paper, they proved that the finiteness of the  5-body problem is generically true, except perhaps when the masses belong to some co-dimensional 2 algebraic varieties. Recently Chang and Chen \cite{chang2024toward} have formulated an algorithm based on the method of \cite{albouy2012finiteness}, towards the finiteness of the 6-body problem. We refer to \cite{albouy2012finiteness, moeckel2014lectures} and the references therein for a more detailed literature review and historical remarks on this problem. 

Our aim in this paper is to investigate the finiteness of central configurations for a potential $-U_\alpha: \mathbb{R}^{2n}\setminus \Delta \rightarrow \mathbb{R}$ which is homogeneous of degree $-\alpha$, i.e., 
\begin{align}
    -U_{\alpha}(q_1,q_2,...,q_n) = 
\begin{cases}
\displaystyle -\sum_{1\leq i<j\leq n}\frac{m_im_j}{\|q_{ij}\|^{\alpha}},\quad &\alpha>0,\\
\displaystyle -\sum_{1\leq i<j\leq n}m_im_j\log \|q_{ij}\|, \quad &\alpha =0.
\end{cases} 
\end{align}
With such homogeneous potentials, the equations of central configurations become 
\begin{align}
\label{eq:cc_homogeneous}
 \forall 1\leq i\leq n,\quad     q_{i} =\sum_{\substack{ 1\leq j\leq n\\j\neq i}}m_i\frac{ q_{ij} }{\|q_{ij}\|^{\alpha+2}}.    
\end{align}
In particular, when $\alpha=1$, $-U_{\alpha}$ becomes the classical Newtonian potential, and when $\alpha=0$, solutions of the central configuration equations are the relative equilibria of the $n$-vortex problem, which is considered as a finite-dimensional approximation of the 2-dimensional incompressible Euler equations in hydrodynamics. 

Note that under the normalisation $\lambda =1$, these configurations $\mathbf{q} =(q_1,q_2,\dots, q_n)$ are critical points of $U_{\alpha}$ restricted to the shape sphere $I(\mathbf{q})=1$. In other words, they are the critical points of the function $U_{\alpha}(\mathbf{q})+\frac{1}{2}I(\mathbf{q})$. We say that a central configuration is \textit{non-degenerate}, if it is a non-degenerate critical point of the above function. It is also worth mentioning that non-degenerate central configurations are also essential ingredients in the construction of highly concentrated solutions of the singular-perturbed elliptic PDEs, for example \cite{CaoPengYan2015,AoWei2021}.

When $\alpha=1$, the central configuration equations form a system of algebraic equations after complexification. As a result, methods from (complex) algebraic geometry can be applied to produce useful consequences, as in \cite{hampton2006finiteness} and \cite{albouy2012finiteness}. When $\alpha$ is a non-negative integer, the algebraic nature of the central configuration equations essentially do not change, and similar analysis applies to such homogeneous potentials, see \cite{hampton2009finiteness} and \cite{yu2023finiteness} and references therein. 

For general $\alpha \geq 0$, however, the central configuration equations will no longer be algebraic equations, and the results seem to be relatively rare. Despite the technical difficulty, such potentials arise naturally in popular models in physics, for both theoretical and applied interests. For example, when $\alpha>2$, the potential corresponds to the $n$-body problem in celestial mechanics with a strong force, where the variational principle can be directly applied to find collision-free periodic orbits. On the other hand, when $\alpha\in (0,2)$, the potential corresponds to the $n$-vortex problem from the generalized surface quasi-geostrophic (gSQG) equation, which models the frontogenesis in meteorology. The study of central configurations for such potentials might provide useful information in understanding these systems. 
\section{Main Results}
Our main result is the following theorem: 
\begin{thm} \label{T:main}
For any $ \alpha\geq 0$ and for any masses $m_1,m_2,...,m_n \in \mathbb{R}_{*}$ satisfying $\sum_{i=1}^n m_i \neq 0$, the number of non-degenerate central configurations with the homogeneous potential $-U_{\alpha}$ is bounded both from above and from below. Moreover, both bounds depend only on the number of bodies $n$, and are independent of the homogeneous degree $\alpha$ and the masses. 
\end{thm}
Note that we do not require $m_i$ be positive. Indeed, a negative $m_i$ can be interpreted as either a negative electron in Coulomb force, or a negative (clockwise) circulation of vortex in hydrodynamics.

We need some new coordinates of the system, which can be used to describe configurations in the quotient space, after the reduction by symmetry. More precisely, we work with the Albouy-Chenciner formulation of \eqref{eq:cc} which is introduced in \cite{albouy1997probleme}. See also \cite{dziobek1900ueber,bolsinov1999lie} for other coordinates constructions.
Let $n \geq 3$ be a given integer and $m_j\neq 0$, $1 \leq j \leq n$ be the prescribed masses such that $\sum_{i=1}^{n}m_i\neq 0$. Let $r_{ij}=\| q_{ij}\|$ and the central configuration equations become  
\begin{equation} \label{eqn:CC-AC-1}
    \sum_{i=1}^n m_i S_{ji} \begin{pmatrix}
    x_{ji}\\
    y_{ji}
    \end{pmatrix} = 0,
\end{equation}
where 
\begin{align}
S_{ji} :=
\begin{cases}
   \displaystyle \frac{1}{r_{ji}^{2+\alpha}} + \bigg(\sum_{k=1}^{n} m_k\bigg)^{-1}, \quad &j\neq i\\
   0,\quad &j= i. 
\end{cases}   
\end{align}
Since the invariance under translation and rotation is indifferent to the homogeneous degree $\alpha$, following the same lines in \cite{albouy1997probleme}, then central configuration equations can be written in a concise matrix form, i.e.,  $BA + AB^{t}=0$, where $A,B\in\mathbb{R}^{n\times n}$ are defined as
\begin{align}
A_{ij}&: = 
\begin{cases}
 m_i S_{ij}, \quad &i\neq j\\
 -\sum_{k\neq j}{A_{kj}}, \quad &j=i.
\end{cases}\\
B_{ij}&:= -\frac{1}{2}\|(x_i,y_i) -(x_j,y_j)\|^2, \quad \quad \hfill \forall 1\leq i<j\leq n.
\end{align}
These equations form a system of $\frac{n^2-n}{2}$ equations, i.e., $\forall 1\leq i<j\leq n$,
\begin{align} 
\label{eq:CC-AC-2}
0 = & \sum_{k=1}^{n} m_k \bigg( 
\bigg(\frac{1}{r_{ik}^{\alpha+2}} + \bigg(\sum_{k=1}^{n} m_k\bigg)^{-1}\bigg)
\bigg(r_{jk}^2 - r_{ik}^2 - r_{ij}^2\bigg) \notag \\
& \quad + 
\bigg(\frac{1}{r_{jk}^{\alpha+2}} + \bigg(\sum_{k=1}^{n} m_k\bigg)^{-1}\bigg)
\bigg(r_{ik}^2 - r_{jk}^2 - r_{ij}^2\bigg)
\bigg).
\end{align}

Obviously, it is no longer a polynomial system, and the locus is not an algebraic set if $\alpha$ is an irrational number. After some manipulation though, these equations consist solely of elementary functions and fit into the frame of fewnomial theory \cite{khovanskii1980class}. In particular, one can apply the following Khovanskii's theorem:
\begin{lem}[Khovanskii\cite{khovanskii1980class}]
Consider a system of $m$ equations $P_1=\ldots P_m=0$, with $n$ real unknowns $\mathbf{x}=x_1,\ldots,x_n$ in which the $P_i$ are polynomials of degree $n_i$ in $n+k$ variables $\mathbf{x},y_1,\ldots,y_k$ where $y_j=\exp{(\mathbf{a}_j\cdot \mathbf{x})}$ for some $\mathbf{a}_j\in \mathbb{R}^n$ $1\leq j\leq k$. The number of non-degenerate solutions of the system is bounded from above by 
\[\displaystyle \prod_{i=1}^m n_i\left(\sum_{i=1}^{m} n_i+1\right)^k2^{(k^2-k)/2}.\]
\end{lem}

\begin{proof}[Proof of theorem 1]
We introduce $\displaystyle \frac{n(n-1)}{2}$ new variables:
\[
z_{ij} =-\ln{ r_{ij}}, 1\leq i<j\leq n.
\]
Now with these new variables, the system is augmented to : 
\begin{equation}
\begin{cases}
   \displaystyle  0=\sum_{k=1}^{n} m_k[ (\exp((\alpha+2) z_{ik}) +(\sum_{k=1}^{n} m_k)^{-1}))(r_{jk}^2-r_{ik}^2-r_{ij}^2) \\
   \quad \quad  +(\exp((\alpha+2) z_{jk}) +(\sum_{k=1}^{n} m_k)^{-1}))(r_{ik}^2-r_{jk}^2 - r_{ij}^2) ] , \\
   \displaystyle  0=r_{ij} - \exp(-z_{ij}),\hfill \forall 1\leq i<j\leq n.
\end{cases}
\end{equation}
The above system consists of $n(n-1)$ unknowns, i.e., $r_{ij}, z_{ij}$, with $ 1\leq i<j\leq n$, as well as $n(n-1)$ equations. Finally, after introducing the auxiliary variables 
\begin{align}
    Y_{ij}&=exp((\alpha+2) z_{ij}),\\
    \tilde{Y}_{ij}&= exp{(-z_{ij})},
\end{align}
the system then becomes
\begin{align}
\begin{cases}
    \displaystyle 0=\sum_{k=1}^{n} m_k[ (Y_{ik} +(\sum_{k=1}^{n} m_k)^{-1})(r_{jk}^2-r_{ik}^2-r_{ij}^2) + \\
    \quad \quad(Y_{jk}  +(\sum_{k=1}^{n} m_k)^{-1})(r_{ik}^2-r_{jk}^2 - r_{ij}^2) ] , \\
     0=r_{ij} - \tilde{Y}_{ij}, \hfill \forall 1\leq i<j\leq n.
\end{cases}
\end{align}
This is a system of $n(n-1)$ polynomials, half of the equations are of degree $3$ and the remaining equations are of degree $1$. By taking $\displaystyle m =k=n(n-1)$ in Khovanskii's theorem, we achieved an upper bound 
\[u(n)=3^{\frac{n(n-1)}{2}}(2n^2-2n+1)^{n(n-1)} 2^{\frac{(n^2-n)(n^2-n-1)}{2}}.\]
For the lower bound, note that central configurations with homogeneous potentials are still the critical points of the potential function $-U_{\alpha}$ on the shape sphere $\mathbb{CP}_{n-2}$. According to the theorem of 
Fadell-Neuwirth \cite{fadell1962configuration} the map $\mathbb{CP}_{n}(\mathbb{R}^2) \rightarrow \mathbb{CP}_{n-1}(\mathbb{R}^2)$ by forgetting the $n^{th}$ point $z_n$ is a fibration with fibers being $\mathbb{R}^2\setminus \{z_1,z_2,...,z_n\}$. As a result, by induction we have the Poincar\'e polynomial 
\begin{align}
    P(t) = (1+t)(1+2t)...(1+(n-1)t).
\end{align}
Taking $t=1$ gives the lower bound \[l(n)=\frac{n!}{2}.\]
\end{proof}

\begin{rem}
In \cite{Hampton2019}, Hampton conjectured that the number of equal-mass central configurations never decreases as the exponent $\alpha$ increases. Our result suggests that if this is the case, then for any fixed $n$, the number of central configurations of the equal-mass $n$-body problem will eventually become a constant, unless degenerate central configurations show up.  
\end{rem}
\begin{rem}
The result above should rather be interpreted as a qualitative consequence, instead of a quantitative one.  Although we have shown the uniform boundedness of the number of central configurations with respect to the degree of homogeneity, neither of these bounds here is optimal one. Concerning the upper bound, for $n=3$, Albouy and Fu \cite{albouy2014some} has sharpened the upper bound to be $3$. By similar method \cite{gabrielov2007mystery}, the number of equilibrium points of a potential with $4$ prefixed point charges in $\mathbb{R}^n$ is bounded by $12$. The lower bounds can be improved via a Morse-theoretic approach, see for instance \cite{Hampton2019,roberts2018morse}. In \cite{hampton2006finiteness}, the number of central configurations of the Newtonian 4-body problem is bounded from above by 8472, which is considerably less than the bound indicated in our theorem.
\end{rem}

\begin{rem}
Note that the upper bound in Theorem \ref{T:main} concerns non-degenerate central configurations only, yet Palmore \cite{Pal1975} has shown that the degenerate central configurations indeed exist for the Newtonian potential. Xia has then proved in \cite{Xia1992} that the set of masses $(m_1,\ldots,m_n)\in \mathbb{R}_+^n$ that permits degenerate central configurations has positive $n-1$ dimensional Hausdorff measure, based on the implicit function theorem. Both statements can be generalised to the homogeneous potentials.     
\end{rem}


\subsection*{Acknowledgment}
This work is supported by the DFG grant ZH 605/1-1 and the NSFC Young Scientists Fund 11901160. We appreciate Lei Zhao for organising the seminar of celestial mechanics at Augsburg University in 2021.

\end{document}